\documentclass[12pt,reqno]{amsart}
\usepackage{amsmath, amsthm, amsopn, amssymb, microtype}
\usepackage{mathrsfs} 
\usepackage{mathscinet}
\usepackage{bbm}
\usepackage{enumerate, tikz, etoolbox, intcalc, geometry, caption, subcaption, afterpage}
\usepackage[english]{babel}

\title{Kac's Lemma and countable generators for actions of countable groups}
\author{Tom Meyerovitch}
\address{Ben Gurion University of the Negev.
	Departement of Mathematics.
	Be'er Sheva, 8410501, Israel
}
\email{mtom@bgu.ac.il}
\author{Benjamin Weiss}
\address{The Hebrew University of Jerusalem.
Einstein Institute of Mathematics
Edmond J. Safra Campus
Givat Ram. Jerusalem, 9190401, Israel
}
\email{weiss@math.huji.ac.il }

\usepackage[colorlinks=true,urlcolor=blue,pdfborder={0 0 0}]{hyperref}
\hypersetup{linkcolor=[rgb]{0,0,0.6}}
\hypersetup{citecolor=[rgb]{0,0.6,0}}
\usepackage[nameinlink]{cleveref}
\crefname{theorem}{Theorem}{Theorems}
\crefname{thm}{Theorem}{Theorems}
\crefname{mainthm}{Theorem}{Theorems}
\crefname{lemma}{Lemma}{Lemmas}
\crefname{lem}{Lemma}{Lemmas}
\crefname{remark}{Remark}{Remarks}
\crefname{prop}{Proposition}{Propositions}
\crefname{defn}{Definition}{Definitions}
\crefname{corollary}{Corollary}{Corollaries}
\crefname{cor}{Corollary}{Corollaries}
\crefname{section}{Section}{Sections}
\crefname{figure}{Figure}{Figures}
\crefname{quest}{Question}{Questions}

\begin{document}

\newtheorem{thm}{Theorem}[section]
\newtheorem{lemma}[thm]{Lemma}
\newtheorem{prop}[thm]{Proposition}
\newtheorem{cor}[thm]{Corollary}
\newtheorem{claim}[thm]{Claim}
\newtheorem{quest}[thm]{Question}
\newtheorem{fact}[thm]{Fact}
\theoremstyle{definition}
\newtheorem{definition}[thm]{Definition}
\newtheorem{example}[thm]{Example}
\newtheorem{remark}[thm]{Remark}
\newtheorem{conj}{Conjecture}

\newcommand{\N}{\mathbb{N}}
\newcommand{\Z}{\mathbb{Z}}
\newcommand{\Q}{\mathbf{Q}}
\newcommand{\R}{\mathbb{R}}
\newcommand{\C}{\mathbf{C}}
\newcommand{{\cP}}{\mathcal{P}}
\maketitle
\begin{abstract}
Kac's lemma determines the expected return time to a set of positive measure under iterations of an  ergodic probability preserving transformations.
    We introduce the notion of an \emph{allocation} for a probability preserving action of a countable group. Using this notion, we formulate and prove generalization of Kac's lemma for an action of a general countable group, and another generalization that applies to probability preserving equivalence relations.
     As an application, we provide a short proof for the existence of  countable generating partitions for  any ergodic action of a countable group.
\end{abstract}
\section{Introduction}
Kac's lemma \cite{MR0022323} 
is a basic result in ergodic theory.
It asserts that the expected number of iterates that it takes a point from a measurable set $A$ to return to the set $A$ under an ergodic probability-preserving transformation is equal to the inverse of the  measure of $A$.
Here is a formal statement:

\begin{thm}[Kac's lemma]\label{thm:Kac_lemma}
Let $(X,\mathcal{B},\mu)$ be a probability space, let
$T:X \to X$ be an ergodic measure preserving transformation of $(X,\mathcal{B},\mu)$ and let $A \in \mathcal{B}$ be a measurable set with $\mu(A) >0$.
Then 
\begin{equation}\label{eq:Kac_1}
\int_A r_A(x) d\mu(x) = 1,    
\end{equation}
where $r_A$ is given by:

\begin{equation}\label{eq:r_A_def}
    r_A(x) = \min \left\{ n \ge 1~:~ T^n(x) \in A \right\}.
\end{equation}
\end{thm}

Poincare's recurrence theorem says that for any probability preserving transformation the function $r_A:A \to \N$ given by \eqref{eq:r_A_def} is  finite almost everywhere on $A$.

The formula \eqref{eq:Kac_1} can be rewritten as 
\[
\frac{1}{\mu(A)}\int_A r_A(x) d\mu(x) = \frac{1}{\mu(A)}.
\]

For $x \in A$ the quantity $r_A(x)$ given by 
\eqref{eq:r_A_def} is the \emph{return time} to $A$ of the point $x$.
The left hand side of the equation is nothing but the expectation of the return time $r_A$ with respect to the conditional probability given $A$.

In this note we formulate and prove some generalizations of Kac's lemma for probability preserving actions of an arbitrary countable group.
We introduce the notion of an \emph{allocation} for a group action:

\begin{definition}\label{def:A_allocation}
Let $(X,\mathcal{B},\mu,T)$ be a  probability-preserving action of a countable group $\Gamma$, and let $A \in \mathcal{B}$ be a measurable set.  An \emph{$A$-allocation} is a measurable function $\kappa:X \to \Gamma$ such that $T_{\kappa(x)}(x) \in A$ $\mu$-almost surely. 

Note that for non-ergodic  probability-preserving actions, there exist sets $A$ of positive measure for which no $A$-allocation exists. For instance, any invariant set $A$ whose complement has positive measure does not admit an $A$-allocation.

Given an $A$-allocation $\kappa:X \to \Gamma$ and $x \in A$, define:
\begin{equation}\label{eq:B_kappa_def}
    B_\kappa(x) = \left\{ \gamma \in \Gamma~:~ \kappa(T_\gamma^{-1}(x))=\gamma \right\}.
\end{equation}
and $T_\kappa:X \to X$ by:
\[
T_\kappa(x) = T_{\kappa(x)}(x).
\]
\end{definition}

\begin{prop}\label{thm:Kacs_lemma_A_allocation}
Let $\Gamma$ be a countable group, let  $(X,\mathcal{B},\mu,T)$ be a probability-preserving $\Gamma$-action, let $A \in \mathcal{B}$ be a measurable set with $\mu(A)>0$.
Suppose that $\kappa:X \to \Gamma$ is an $A$-allocation.
Given a $\mathcal{B}$-measurable function  $f:X\to  [0,+\infty]$ define
\[
f_\kappa:A \to [0,+\infty]
\]
as follows:
\begin{equation}\label{eq:f_kappa_def}
    f_\kappa(x)= \sum_{y \in X ~:~ T_\kappa (y) = x} f(y).
\end{equation}
Then
\begin{equation}
    \int_A f_\kappa(x) d\mu(x) = \int_X f(x) d\mu(x).
\end{equation}

In particular, taking $f \equiv 1$ we have that
\begin{equation}\label{eq:A_alloc_exp_cell_size}
    \int_A | B_\kappa(x)| d\mu(x) = 1,
\end{equation}
where $B_\kappa(x)$ is given by \eqref{eq:B_kappa_def}.
\end{prop}

As a special case,  if $(X,\mathcal{B},\mu,T)$ is an ergodic probability-preserving action of the group $\Z$ (namely, an \emph{ invertible} transformation) and $A \in \mathcal{B}$ is a set of positive measure then  
 the set $\bigcup_{n=1}^\infty T^{-n}(A)$ has  full measure whenever $A \in \mathcal{B}$ has positive measure, and so we can define an $A$-allocation 
$\kappa:X \to \Z$ given by
\[\kappa(x) = \min \left\{ n \ge 0~:~ T_n(x) \in A\right\}.\]
In this case $|B_\kappa(T_A(x))| =r_A(x)$ given by \eqref{eq:r_A_def}, 
where $T_A:A \to A$ is the \emph{first return transformation} given by
\[
T_A(x)= T^{r_A(x)}(x).
\]

Using the fact that the first return transformation $T_A:A \to A$ is measure-preserving, in this case \eqref{eq:A_alloc_exp_cell_size} reduces to Kac's formula \eqref{eq:Kac_1}.

A somewhat different variant of  Kac's lemma for ergodic probability preserving actions of $\Z^d$ was considered  
in  \cite{MR1412593} by Aaronson and Weiss: Given a probability-preserving $\Z^d$-action $(X,\mathcal{B},\mu,T)$  
and a subset $A \in \mathcal{B}$ with $\mu(A)>0$,
Aaronson and Weiss defined a \emph{Kac function} for the set $A$ to be a measurable function $\phi:A \to \N$ such that
\[   X = \bigcup_{n=1}^\infty \bigcup_{v \in B_n}T_v(\phi^{-1}(\{n\})) ~\mod \mu,
\]
and
\[
\int_A | B_{\phi(x)}| d\mu(x) < + \infty,
\]
where $B_n = [-n,n]^d \cap \Z^d$.
In this terminology, for the case $d=1$, Kac's lemma implies that the return time function $r_A:A \to \N$ is a Kac function. Aaronson and Weiss proved that for any ergodic  probability-preserving $\Z^d$-action $(X,\mathcal{B},\mu,T)$ and any  subset $A \in \mathcal{B}$ with $\mu(A)>0$, the product of $\Z^d$-action $(X,\mathcal{B},\mu,T)$ with a dyadic odometer admits a Kac function for the lift of the set $A$ to the product system \cite[Theorem 2.1]{MR1412593}.
Note that in \cite{MR1412593} the inequality $\int_A | B_{\phi(x)}| d\mu(x) < + \infty$ was stated in the equivalent form   $\int_A (\phi(x))^d d\mu(x) < +\infty$.

In fact, the proof given in \cite{MR1412593} shows that for every fixed $d \ge 1$ there is a constant $C_d \in (0,+\infty)$ so that any subset $A \in \mathcal{B}$ with $\mu(A)>0$ admits a Kac function with $\int_A | B_{\phi(x)}| d\mu(x) < C_d$. The proof in \cite{MR1412593} is written explicitly for the case $d=2$, where the constant obtained is $C_d=36$. In the case $d=1$, taking $B_n=[0,n-1] \cap \Z$, the statement of  Kac's lemma shows that the return time function is a Kac's lemma with an optimal constant  $C_1=1$. 


Along these lines, the following is a generalization of Kac's lemma for actions of arbitrary countable groups:

\begin{definition}\label{def:sweep_out}
   Let $(X,\mathcal{B},\mu,T)$ be a probability preserving action of a countable group $\Gamma$. A set $A\in \mathcal{B}$ is called a \emph{sweep-out set} if the orbit of almost every $x\in X$ intersects $A$. Equivalently, $A \in \mathcal{B}$ is a sweep-out set if 
   \[
   \mu\left(\bigcup_{\gamma \in \Gamma}T_{\gamma}(A)\right)=1.
   \]
\end{definition}

A probability preserving $\Gamma$-action   $(X,\mathcal{B},\mu,T)$ is  \emph{ergodic} exactly when any set $A \in \mathcal{B}$ with $\mu(A)>0$ is a sweep-out set.

\begin{thm}\label{thm:Kac_function_group_action}
Let $\Gamma$ be a countable group. Then there exists a sequence of finite subsets of $\Gamma$ $B_1,B_2,\ldots,B_n,\ldots \subset \Gamma$ with $1_\Gamma \in B_n$ for all $n$
such that for any   probability-preserving $\Gamma$-action $(X,\mathcal{B},\mu,T)$ and any  sweep-out set $A \in \mathcal{B}$,  there exists 
 a measurable function $\phi:A \to \N$ such that 
\begin{equation}\label{eq:X_partition}
 X = \biguplus_{n=1}^\infty \biguplus_{\gamma \in B_n}T_\gamma^{-1}(\phi^{-1}(\{n\})) ~\mod \mu,
\end{equation}
and
\begin{equation}\label{eq:int_Kac_function_is_1}
    \int_A |B_{\phi(x)}| d\mu(x) =1 
\end{equation}
\end{thm}

\Cref{eq:int_Kac_function_is_1} is a strengthening of the condition $\int_A | B_{\phi(x)}| d\mu(x) < + \infty$  in the Aaronson-Weiss $\Z^d$ random Kac inequality as in \cite[Theorem 2.1]{MR1412593}. However, \Cref{thm:Kac_function_group_action} does not recover the 
 Aaronson-Weiss $\Z^d$ random Kac inequality, since the sequence of sets $B_n$ given by the proof cannot in general be taken to be  $\Z^d$-boxes. 
 
Under additional assumptions on the group $\Gamma$, it is possible to say more about the sequence of sets $(B_n)_{n=1}^\infty$.

\begin{definition}
We say that  $B\subset\Z^d$ is \emph{almost convex} if $B$ contains every integer point in the interior of its convex hull.
\end{definition}

\begin{prop}\label{prop:Kac_Zd_convex}
 When $\Gamma= \Z^d$, in the statement of \Cref{thm:Kac_function_group_action} it is possible to choose the set sequence of subsets $(B_n)_{n=1}^\infty$ to be finite almost convex subsets in $\Z^d$ that contain zero.
\end{prop}

In the case $\Gamma= \Z$, an almost convex subset is  always of the form $I \cap \Z$, where $I \subset \R$ is an interval.
So \Cref{prop:Kac_Zd_convex} essentially recovers Kac's lemma for this case.

We now turn to formulate a version of Kac's lemma for probability-preserving equivalence relations.

Given a  probability space  $(X,\mathcal{B},\mu)$,
A \emph{countable $\mu$-preserving equivalence relation} a is  binary  equivalence relation $\mathcal{R} \in \mathcal{B} \otimes \mathcal{B}$ with countable or finite equivalence classes, meaning that
for every $x \in X$ the set
\[
\left\{ y \in X~:~  (x,y) \in \mathcal{R}\right\}
\]
is at most countable, such that for any $A \in \mathcal{B}$ and any $\mathcal{B}$-measurable bijection $f:X \to X$ whose graph $\{(x,f(x)~:~ x \in X\}$ is contained in $\mathcal{R}$ we have $\mu(A)=\mu(f^{-1}(A))$. In particular, for any  such $f:X\to X$ we have that $A \subseteq X$ is $\mu$-null if and only if $f^{-1}(A)$ is $\mu$-null.

The typical example is the orbit equivalence relation $\mathcal{R}_T$ of a probability-preserving action $(X,\mathcal{B},\mu,T)$ of a countable group $\Gamma$, given by
\[ (x,y)\in \mathcal{R}_T ~\Leftrightarrow ~ \exists \gamma \in \Gamma \mbox{ s.t } y= T_\gamma(x).\]
We refer to Feldman and Moore's paper \cite{MR578656} for background on the ergodic theory of equivalence relations.
Using classical results on the ergodic theory of equivalence relations and \Cref{thm:Kacs_lemma_A_allocation}, we obtain the following.

\begin{thm}[Kac's lemma for probability preserving equivalence relations]\label{thm:Kac_orbit_relation}
    Let $(X,\mathcal{B},\mu)$ be a standard probability space and let $\mathcal{R} \subset X \times X$ be a countable $\mu$-preserving equivalence relation.
    Suppose that $\tau:X \to X$ is a $\mathcal{B}$-measurable function such that $(x,\tau(x)) \in \mathcal{R}$ $\mu$-almost everywhere.
    For $f:X \to [0,+\infty]$ define 
    $f_\tau:X \to [0,+\infty]$ by
    \begin{equation}\label{eq:f_tau_def}
    f_\tau(x) = \sum_{y \in \tau^{-1}(\{x\})}f(y).
    \end{equation}
    Then
    \begin{equation}\label{eq:f_tau_Kac}
    \int_X f(x) d\mu(x)= \int_X f_\tau(x) d\mu(x).    
    \end{equation}
    In particular, taking $f \equiv 1$ we get:
    \begin{equation}\label{eq:preimages_exp}
        \int_X \left| \tau^{-1}( \{x\}) \right| d\mu(x)=1.
    \end{equation}
\end{thm}

\Cref{eq:preimages_exp} says that the expected number of preimages of a point under a self-map that respects a countable $\mu$-preserving equivalence relation is exactly one,  consistently with the common intuition.

Note that when $(X,\mathcal{B},\mu)$ is a \emph{finite} probability space, \Cref{eq:preimages_exp}  is  completely trivial: This case amounts to the trivial observation  that for any function from a finite set $X$ to itself, the expected number of preimages of a uniformly chosen random element of $X$ is equal to $1$.

Also, in the special case where $\tau:X \to X$ is \emph{injective}, \Cref{eq:f_tau_Kac} reduces to the statement that $\tau$ is a $\mu$-preserving transformation, recovering the well-known fact that any element of the full-group of a probability-preserving equivalence relation $\mathcal{R}$ is itself a probability-preserving transformation.

In the special case where $T:X \to X$ is an invertible ergodic probability-preserving transformation of $(X,\mathcal{B},\mu)$, $A \in \mathcal{B}$ is a set of positive measure  and 
$\tau:X \to X$ is the map that sends $x \in X$ to $T^{-n_A(x)}(x)$, where $n_A(x)$ is the smallest positive integer $m$ such that $T^{-m}(x) \in A$, we have that 
\[
|\tau^{-1}(\{x\})| = \begin{cases}
    r_A(x) & x\in A\\
    0 & x \not \in A,
\end{cases}
\]
and so \Cref{eq:preimages_exp} reduces to \Cref{eq:Kac_1}.

\section{Proof of the generalized Kac lemma for $A$-allocations}
In this section we prove \Cref{thm:Kacs_lemma_A_allocation}.
We begin with the following basic lemma:

\begin{lemma}\label{lem:Kac_int_upperbound}
Let  $(X,\mathcal{B},\mu,T)$ be a  probability-preserving action of a countable group $\Gamma$, let $B_1,\ldots,B_n,\ldots \subset \Gamma$ be a sequence of finite subsets of $\Gamma$ and let $A \in \mathcal{B}$ be a measurable set.  
Suppose that $\phi:A \to \N$ is a measurable function such that 
for every $n,m \in \N$, $\gamma \in B_n$ and $\tilde \gamma \in B_m$ we have 
\[ T_\gamma^{-1}(\phi^{-1}(\{n\})) \cap T_{\tilde \gamma}^{-1}(\phi^{-1}(\{m\})) = \emptyset,\]
unless $m=n$ and $\gamma=\tilde \gamma$.
Given a non-negative measurable function $f:X \to [0,\infty]$ define $f_\phi:A \to [0,\infty]$  by

\begin{equation}\label{eq:f_phi_def}
    f_\phi(x) = \sum_{\gamma \in B_{\phi(x)}} f(T_\gamma(x)).
\end{equation}
Then 
\[
\int_A f_\phi(x) d\mu(x) \le \int_X f(x) d\mu(x).
\]
If furthermore the union of the sets $T_\gamma^{-1}(\phi^{-1}(\{n\}))$ over all $n \in \N$ and $\gamma \in B_n$ is a set of full measure (namely if \eqref{eq:X_partition} holds), then 
\begin{equation}\label{eq:f_phi_int}
    \int_A f_\phi(x) d\mu(x) = \int_X f(x) d\mu(x).
\end{equation}
\end{lemma}

\begin{proof}
We can rewrite the function $f_\phi:A \to [0,+\infty]$ as follows:
\[
f_\phi(x) = \sum_{n=1}^\infty  1_{\phi^{-1}(\{n\})}(x) \sum_{\gamma \in B_n} f(T_\gamma(x)).
\]
Since the sets  $T_\gamma^{-1}(\phi^{-1}(\{n\}))$ are pairwise disjoint, we have 
\[
\int_X f(x)d\mu(x) \ge \sum_{n=1}
^\infty \sum_{\gamma \in B_n} \int_{T_\gamma^{-1}(\phi^{-1}(\{n\}))}f(x) d\mu(x) \]    
Since $T_\gamma:X \to X$ is $\mu$-preserving, we have that
\[
\int_{T_\gamma^{-1}(\phi^{-1}(\{n\})} f(x) d\mu(x) = 
\int_{\phi^{-1}(\{n\})}f(T_\gamma(x)) d\mu(x).
\]
It follows that
\[
\int_X f(x)d\mu(x) \ge \sum_{n=1}^\infty\sum_{\gamma \in B_n} \int_{\phi^{-1}(\{n\})}f(T_\gamma(x)) d\mu(x) =
\int_A f_\phi(x) d\mu(x),
\]
where in the last inequality we used that $A= \biguplus_{n=1}^\infty \phi^{-1}(\{n\})$.
This proves that $\int_A f_\phi(x) d\mu(x) \le \int_X f(x) dx$.
If furthermore \eqref{eq:X_partition} holds, then the inequality in the first displayed equation of the proof is an equality, and so in this case 
$\int_A f_\phi(x) d\mu(x) = \int_X f(x) dx$.
\end{proof}

The following lemma is a variation of Poincare's recurrence theorem, namely the statement that the first return time $r_A$ is finite almost-everywhere on $A$:
\begin{lemma}\label{lem:Kac_function_set_finite}
    Let  $(X,\mathcal{B},\mu,T)$ be a  probability-preserving action of a countable group $\Gamma$, let $B_1,\ldots,B_n,\ldots \subset \Gamma$ be a sequence of  subsets of $\Gamma$ and let $A \in \mathcal{B}$ be a measurable set.  
Suppose that $\phi:A \to \N$ is a measurable function such that the sets 
\[ T_\gamma^{-1}(\phi^{-1}(\{n\})),~ n \in \N,~ \gamma \in B_n\]
are pairwise disjoint.
Then for any $n \in \N$ we have
\[\mu(\{ x \in A~:~ |B_{\phi(x)}| \ge n\}) \le \frac{1}{n}.\]
In particular, if $|B_n|=+\infty$ then $\mu(\phi^{-1}(n))= 0$.
\end{lemma}
\begin{proof}
    Applying \Cref{lem:Kac_int_upperbound} with  $f \equiv 1$ we have 
    $ \int_A |B_{\phi(x)}|=1$. 
So by Markov inequality we have
\[\mu(\{ x \in A~:~ |B_{\phi(x)}| \ge n\}) \le \frac{1}{n}.\]
\end{proof}

To relate the previous statement directly  with $A$-allocations we have the following:
\begin{lemma}\label{prop:A_allocation_partition}
    Let $\Gamma$ be a countable group. Then there exists a sequence of finite subsets $B_1,B_2,\ldots,B_n,\ldots \subset \Gamma$ so that for any
      probability-preserving action $(X,\mathcal{B},\mu,T)$ of $\Gamma$, any  sweep-out set $A \in \mathcal{B}$ and any  $A$-allocation $\kappa:X \to \Gamma$  we have
\begin{equation}\label{eq:A_partition_B_kappa}
A= \biguplus_{n=1}^\infty \left\{ x \in A~:~  B_\kappa(x) =  B_n \right\} \mod \mu, 
\end{equation}
where $B_{\kappa}(x) \subseteq \Gamma$ is given by \eqref{eq:B_kappa_def}

Moreover, the almost-everywhere defined function $\phi:A \to \N$ given by 
\begin{equation}\label{eq:phi_kappa_def}
\phi(x)=  n  \mbox{ if } B_\kappa(x) = B_n 
\end{equation}
satisfies \eqref{eq:X_partition}.
\end{lemma}
We remark that in the statement of \cref{prop:A_allocation_partition} the sweep-out property of the set $A$ is only used to assure the existence of an $A$-allocation. 
If $A$ is not a sweep-out set, there cannot exist an $A$-allocation because in that case the set $B:= X \setminus \bigcup_{\gamma \in \Gamma}T_{\gamma}(A)$ has positive measure.
In fact, as shown in the proof of \Cref{thm:Kac_function_group_action} below, any sweep-out set $A$ admits and $A$-allocation.

\begin{proof}
    We can partition the set $A$ according to the value of $B_\kappa(x)$:
    \[A = \biguplus_{F \subseteq \Gamma} \{ x \in A~:~ B_\kappa(x) =F\}.\]
    By \Cref{lem:Kac_function_set_finite}, $B_\kappa(x)$ is almost-surely a finite set.
    Since there are countably many finite subsets in $\Gamma$, we can enumerate them as $B_1,\ldots,B_n,\ldots$, and
    get \eqref{eq:A_partition_B_kappa}.
    Since $\kappa:X \to \Gamma$ is an $A$-allocation we have that $T_{\kappa(x)}(x) \in A$ almost surely so the function $\phi:A \to \N$ in the statement indeed satisfies \eqref{eq:X_partition}.
\end{proof}


\begin{proof}[Proof of \Cref{thm:Kacs_lemma_A_allocation}]
  Let $(X,\mathcal{B},\mu,T)$ be a probability-preserving $\Gamma$-action, $A \in \mathcal{B}$ with $\mu(A)>0$ and $\kappa:X \to \Gamma$ an $A$-allocation. By \Cref{prop:A_allocation_partition} there exists a sequence of finite subsets $B_1,B_2,\ldots,B_n,\ldots \subset \Gamma$ so that the function $\phi:A \to \N$ defined by \eqref{eq:phi_kappa_def} satisfies \eqref{eq:X_partition}.
  Hence by, \Cref{lem:Kac_int_upperbound}, for every measurable $f:X\to [0,+\infty]$  \Cref{eq:f_phi_int} holds, where $f_\phi:A \to [0,+\infty]$ is given by \Cref{eq:f_phi_def}.
  Observe that in this case $B_{\phi(x)}(x)=B_\kappa(x)$. It follows that $f_\phi =f_\kappa$, where $f_\kappa$ is given by \eqref{eq:f_kappa_def}.
  This completes the proof.
\end{proof}

\begin{proof}[Proof of \Cref{thm:Kac_function_group_action}]
    Let $\Gamma$ be a countable group.
    Since $\Gamma$ is countable, we can choose a bijection $g:\N \to \Gamma$ such that $g(1)=1_\Gamma$ (we follow the convention  $\N=\{1,2,3,\ldots\}$). 
    Let $B_1,\ldots,B_n,\ldots$ be a sequence of finite subsets of $\Gamma$ that satisfies the conclusion of \Cref{prop:A_allocation_partition}. Let 
     $(X,\mathcal{B},\mu,T)$ be a  probability-preserving action of the group $\Gamma$, let $A \in \mathcal{B}$ be a sweep-out set.
    By definition of a sweep-out set, the complement of $\bigcup_{\gamma \in \Gamma}T_\gamma(A)$ is null. So for $\mu$-almost every $x \in X$ there exists $\gamma \in \Gamma$ such that $T_\gamma(x) \in A$.
    Thus, we can define $\kappa:X \to \Gamma$ almost everywhere by:
    \[
    \kappa(x) = g(\ell(x)), 
    \]
    where $\ell:X \to \N$ is defined (almost everywhere) by:
    \[
    \ell(x) = \min \{ n \in \N~:~ T_{g(n)}(x) \in A\}.
    \]
    Then $\kappa:X \to \Gamma$ is an $A$-allocation.
    
    Note that $\kappa:X \to \Gamma$ is a measurable function because
    for every $\gamma \in \Gamma$ there exists a unique $j \in \N$ such that $g(j)=\gamma$ and so
    \[
    \kappa^{-1}(\{\gamma\})=\ell^{-1}(\{j\}) = T_{\gamma}^{-1}(A)\setminus \bigcup_{\ell=1}^{j-1}T^{-1}_{g(\ell)}(A) \in \mathcal{B}.
    \]
    
    Furthermore, because $g(1)=1_\Gamma$ we have that $\ell(x)=1$ for every $x \in A$ so $\kappa(x)=1_\Gamma$ for every $x \in A$.
    
    So by \Cref{prop:A_allocation_partition}
    the function $\phi:A \to \N$ given by \eqref{eq:phi_kappa_def} satisfies \eqref{eq:X_partition}. Hence by \Cref{lem:Kac_int_upperbound} we have that \eqref{eq:f_phi_int} holds for any $\mathcal{B}$-measurable function $f:X \to [0,+\infty]$. In particular, setting $f \equiv 1$, we get \eqref{eq:int_Kac_function_is_1}. Note that since $\kappa(x)=1_\Gamma$ for every $x \in A$, we have that $1_\Gamma \in B_{\phi(x)}$ for $\mu$-almost every $x \in X$, so we can arrange that $1_\Gamma \in B_n$ for every $n \in\N$ by modifying some of the $B_n$'s in a manner that does not change $B_{\phi(x)}$ outside a null set.
\end{proof}
\section{A Kac's lemma for $\Z^d$ with almost-convex shapes}

In this section we prove \Cref{prop:Kac_Zd_convex}.

Formally, proof proceeds by carefully selecting the bijection $g:\Gamma \to \N$ that appears in the proof of \Cref{thm:Kac_function_group_action}. 

\begin{proof}[Proof of \Cref{prop:Kac_Zd_convex}]
Let $(X,\mathcal{B},\mu,T)$ be a probability-preserving $\Z^d$-action.
Choose the bijection 
$g:\N \to \Gamma$ in the proof of \Cref{thm:Kac_function_group_action} so that
the function $n \mapsto \|g(n)\|$ is monotone non-decreasing, where $\|v\|$ denotes the Euclidean norm of $v \in \Z^d$.
The application of \Cref{prop:A_allocation_partition} inside the proof of \Cref{thm:Kac_function_group_action} shows that we can choose the sequence $B_1,\ldots,B_n \subset \Z^d$ to consist only of finite sets $F \subset \Z^d$ such that
\[
\mu \left( \left\{ x \in A~:~ B_\kappa(x) = F \right\}\right) >0,
\]
where for $x \in A$ the set $B_\kappa(x) \subseteq \Z^d$ is given by \eqref{eq:B_kappa_def}.
To complete the proof, we will show that the set $B_\kappa(x) \subseteq \Z^d$ is always almost convex, under the assumption that the function $n \mapsto \|g(n)\|$ is monotone non-decreasing.

Applying the formula for $\kappa:X \to \Gamma=\Z^d$ in the proof of \Cref{thm:Kac_function_group_action}
we see that for every $x \in A$ we have
\begin{equation}\label{eq:B_kappa_g}
B_\kappa(x) = \left\{ v \in \Z^d~:~  T_{g(n)}(T_v^{-1}(x)) \not \in A \mbox{ for all } n< g^{-1}(v) \right\}.    
\end{equation}
Define a total order $\prec$ on the group $\Z^d$ by 
\[
v_1 \prec v_2 \Leftrightarrow 
g^{-1}(v_1) <  g^{-1}(v_2).
\]
Then for every $x \in A$ \eqref{eq:B_kappa_g} can be rewritten as follows:
\[
B_\kappa(x)= \left\{ v \in \Z^d ~:~ v \preceq w+v \mbox{ for all } w \in W_x \right\}
\]
where for $x \in A$ the set $W_x \subseteq \Z^d$ is given by:

\[
W_x =\left\{w \in \Z^d~:~ T_w(x) \in A \right\}.
\]

For every $x \in A$ let:
\[
B^{<}(x) = \left\{ v \in \Z^d~:~ \|v\| < \|w+v\| ~\forall w \in W_x \setminus \{0\} \right\}.
\]
and
\[
B^{\le}(x) =  \left\{ v \in \Z^d~:~ \|v\| \le \|w+v\| ~\forall w \in W_x  \right\}.
\]

Because  $n \mapsto \|g(n)\|$ is non-decreasing we have that
\[
B^{<}(x) \subseteq B_\kappa(x) \subseteq B^{\le}(x).
\]

The reader can check that the set $B^{\le}(x)$ is  the set of integer points of the Voronoi cell of $0$ with respect to the Voronoi tiling of $\R^d$ associated with  the configuration $-W_x \subset \R^d$. The set $B^{<}(x)$ consists of the integer points in the interior of the corresponding Voronoi cell. Informally, the set   $B_\kappa(x)$ can be obtained by ``slightly perturbing'' the Voronoi tiling  associated with  the configuration $-W_x \subset \R^d$ so that no integer points lay on the boundary, and $B_{\kappa}(x)$ to be the perturbed cell of $0$.

The fact that  Voronoi tilings in Euclidean space consist of convex polygons is a classical fact in euclidean geometry.
If $v \in \Z^d$ is in the interior of the convex hull of $B^{\le}(x)$ then
it is certainly in the interior of the Voronoi cell so $v \in B^{<}(x)$.
It follows that any integer point in the the interior of the convex hull of $B_\kappa(x)$ is contained in $B_\kappa(x)$, so $B_{\kappa}(x)$ is almost convex.
As in the proof of \Cref{thm:Kac_function_group_action}, since $\kappa(x) =0$ for every $x \in A$, we have that $0 \in B_{\kappa}(x)$ for $\mu$-almost every $x \in X$.

This completes the proof of \Cref{prop:Kac_Zd_convex}.
\end{proof}


\section{Proof of  Kac's formula for probability preserving equivalence relations}


\begin{proof}[Proof of \Cref{thm:Kac_orbit_relation}]
 Let $(X,\mathcal{B},\mu)$ be a standard probability space and let $\mathcal{R} \subset X \times X$ be a countable $\mu$-preserving equivalence relation, and let $\tau:X \to X$ be a $\mathcal{B}$-measurable function such that $(x,\tau(x)) \in \mathcal{R}$ $\mu$-almost everywhere. As noted by Feldman and Moore \cite{MR578656}, there exists a countable group $\Gamma$ and a probability preserving $\Gamma$-action $(X,\mathcal{B},\mu,T)$ such that $\mathcal{R}$ is the orbit relation of $T$.
 Enumerate the elements of $\Gamma$ as 
 \[\Gamma = \{\gamma_1,\gamma_2,\ldots,\},\]
 and define $\kappa:X \to \Gamma$ by 
 \[
 \kappa(x) = \gamma_{m(x)} \mbox{ where } m(x)= \min\{ n \in \N~:~ T_{\gamma_n}(x) = \tau(x)\}.
 \]
As in the proof of \Cref{thm:Kac_function_group_action}, we see that $\kappa$ is a measurable function.

 Define $A= \tau(X)$.
 Then $A \in \mathcal{B}$ because
 \[
 A= \bigcup_{\gamma \in \Gamma} T_\gamma(\{ x\in X~:~\kappa(x) = \gamma\}),
 \]
 and
 $\kappa:X \to \Gamma$ is an $A$-allocation, and $\tau = T_\kappa$.
 Let  $f:X \to [0,+\infty]$ be a  $\mathcal{B}$-measurable function.
 Then 
 \[\int_X \sum_{y \in \tau^{-1}(\{x\})}f(y) d\mu(x)= \int_A f_\kappa(x) d\mu(x).\]
 By \Cref{thm:Kacs_lemma_A_allocation} we conclude that 
 \[
 \int_X \sum_{y \in \tau^{-1}(\{x\})}f(y) d\mu(x)= \int_X f(x) d\mu(x).
 \]

\end{proof}

\section{Countable generators for ergodic group actions}

In this section we use the existence and basic properties of $A$-allocations to provide a short proof that any ergodic probability-preserving action of a countable group $\Gamma$ on a standard probability space admits a countable generator. 

Actually, we prove the existence of finite generator in a more general setting. To deal with non-ergodic actions, we need the following lemma regarding the existence of countable partitions into sweep-out sets:


\begin{lemma}\label{lem:sweep_out_partition}
    Let $(X,\mathcal{B},\mu,T)$ be a probability preserving action of the countable group $\Gamma$ on a standard  probability space $(X,\mathcal{B},\mu)$. Assume that almost every point in $X$ has an infinite orbit. 
    Then for any $\epsilon >0$ there exists a countable collection of pairwise disjoint sweep-out subsets of  $X$, each of which has measure at most $\epsilon$.
\end{lemma}
\begin{proof}
    Since $(X,\mathcal{B},\mu)$ is a standard probability space, we can assume by  Carath{\'e}odory's theorem that $X=[0,1]$,  that $T$ is a $\Gamma$-action on $[0,1]$ by Borel automorphisms that $\mu$ is a  $\Gamma$-invariant Borel-measure on $[0,1]$ and that $\mathcal{B}$ is the $\mu$-completion of the Borel $\sigma$-algebra on $[0,1]$.
    Let $\mathcal{I} \subseteq \mathcal{B}$ denote the $\sigma$-algebra of $T$-invariant sets. Given $A \in \mathcal{B}$, the conditional probability $\mu(A \mid \mathcal{I})(x)$ is defined $\mu$-almost everywhere, and $x \mapsto \mu(A \mid \mathcal{I})(x)$ is a $\mu$-measurable  almost everywhere defined function.

Fix $0 < \epsilon < 1$.

Let 
\[
f_1(x) = \inf \{ t >0~:~ \mu( [0,t] \mid \mathcal{I})(x) \ge \epsilon/2\}.
\]
The almost-everywhere defined function $f_1:X \to [0,1]$ is $\mathcal{I}$-measurable, thus $\mathcal{B}$-measurable.

Since almost every point in $X$ has an infinite orbit,  we conclude that almost surely the conditional probability $\mu(\cdot \mid \mathcal{I})(x)$ has no atoms, and so  for almost every $x \in X$ the function $t \mapsto  \mu([0,t] \mid \mathcal{I})(x)$ is continuous. Thus $\mu([0,f_1(x)] \mid \mathcal{I})(x) = \epsilon/2$ for $\mu$-almost every $x \in X$. In particular, since $\epsilon /2 < 1$, it follows that $f_1(x) < 1$ almost everywhere.
Similarly, by induction  we can define for every $n \ge 2$ an almost everywhere defined measurable function  $f_{n}:X \to [0,1]$ by
\[
f_n(x) = \inf \{ t >f_{n-1}(x) ~:~ \mu( [f_{n-1}(x),t) \mid \mathcal{I})(x) \ge \epsilon/2^n\}.
\]
Similarly to the case $n=1$, using that $\mu(
\cdot 
\mid \mathcal{I})$ is nonatomic almost surely, we have that  $\mu$-almost everywhere
\[ 0 < f_1(x) < f_2(x) < \ldots < f_n(x) < \ldots < 1,\]
and also for  $\mu$-almost every $x \in X$ we have
\[ \mu\left( (f_{n-1}(x),f_n(x)] \mid \mathcal{I}\right)(x) = \frac{\epsilon}{2^n}. \]
Let
\[
A_1 = \{ x \in [0,1]~:~ x \le f_1(x)\}
\]
and for every $n \ge 2$
\[
A_n = \{ x \in [0,1]~:~ f_{n-1}(x) < x \le f_n(x)\}.
\]
Then $A_1,\ldots,A_n,\ldots$ are pairwise disjoint $\mathcal{B}$-measurable sets with $\mu(A_n) = \frac{\epsilon}{2^n}$.
It remains to check that each $A_n$ is a sweep-out set.

Indeed, since $\bigcup_{\gamma \in \Gamma}T_\gamma(A_n) \in \mathcal{I}$ it follows that $\mu$-almost everywhere we have

$\mu( \bigcup_{\gamma \in \Gamma}T_\gamma(A_n) \mid \mathcal{I}) (x)\in \{0,1\}$.

Since $\mu( \bigcup_{\gamma \in \Gamma}T_\gamma(A_n) \mid \mathcal{I}) (x) \ge \epsilon/2^n$ almost-everywhere, we conclude that indeed
 each $A_n$ is a sweep-out set.

\end{proof}

\begin{definition}
Let $(X,\mathcal{B},\mu,T)$ be a probability preserving action of the countable group $\Gamma$ on a standard probability space $(X,\mathcal{B},\mu)$.  
A  countable $\mu$-measurable partition $\mathcal{P} \subseteq \mathcal{B}$  is called a \emph{generator} for the action if 
the smallest $\mu$-complete $T$-invariant $\sigma$-algebra that contains $\mathcal{P}$ is equal to $\mathcal{B}$.
\end{definition}

\begin{thm}\label{thm:countable_generator_group_action}
Let $(X,\mathcal{B},\mu,T)$ be a probability preserving action of the countable group $\Gamma$ on a standard probability space $(X,\mathcal{B},\mu)$, so that almost every point in $X$ has an infinite orbit.  Then there exists a countable generator for the action $(X,\mathcal{B},\mu,T)$.
\end{thm}

\begin{proof}
    Let $E_1,\ldots,E_n,\ldots \in \mathcal{B}$ be a sequence of elements that generates the $\sigma$-algebra $\mathcal{B}$, in the sense that $\mathcal{B}$ is the smallest $\mu$-complete   $\sigma$-algebra containing $\{E_n\}_{n=1}^\infty$.
    We assume that the measure $\mu$ is continuous, otherwise by ergodicity it is supported on a finite orbit, in which case the existence of a countable partition  $\mathcal{P}$   is trivial.
    
    By \Cref{lem:sweep_out_partition} there exists a sequence of pairwise disjoint sets $A_1,A_2,\ldots,A_n,\ldots \in \mathcal{B}$ such that each $A_n$ is a sweep-out set.
    For every $n \in \N$, let $\kappa_n:X \to \Gamma$ be an $A_n$-allocation.
    For every $x \in X$, 
    consider the set $C_x \subseteq \N \times \Gamma $ given by
    \[
    C_x = \left\{ (n,\gamma) \in \N \times \Gamma  ~:~ \gamma \in B_{\kappa_n}(x) \mbox{ and } T^{-1}_\gamma(x) \in E_n  \right\},
    \]
    where $B_\kappa(x)$ is given by \eqref{eq:B_kappa_def}.

    We claim that $C_x$ is a finite set $\mu$-almost surely.
    Indeed,  $(n,\gamma) \in C_x$ implies that $x \in A_n$ and that $\gamma \in B_{\kappa_n}(x)$. Since $B_\kappa(x)$ is a finite set almost surely for any $A$-allocation, it follows that $C_x$ is finite almost surely.

    For every finite set $D \subset \N \times \Gamma$ let 
    \[
    P_D = \{ x \in X ~:~ C_x = D\}.
    \]
    Then $\mathcal{P}:= \{P_D\}_{D \Subset \N \times \Gamma}$ is a countable partition of $X$ minus a null set.
    
    For any $x \in X$ and $n \in \N$ we have that
    $x \in E_n$ if and only if $(n,\gamma) \in C_{T_\gamma(x)}$.
    It follows that
    \[
    E_n = \bigcup_{\gamma \in \Gamma} \bigcup_{D \ni (n,\gamma) }T_\gamma^{-1}(P_D) \mod \mu,
    \]
    where the inner union is over all finite subsets $D \subset \N \times \Gamma$ such that $(n,\gamma) \in D$.

    This proves that every set $E_n$ belongs to the smallest $\mu$-complete $\sigma$-algebra that contains $\mathcal{P}$, hence $\mathcal{P}$ is a generator for the action $(X,\mathcal{B},\mu,T)$

\end{proof}

\begin{cor}
A   probability-preserving action $(X,\mathcal{B},\mu,T)$ of a countable group  $\Gamma$ on a standard probability space admits a countable generator if and only if the set $X_0 \subseteq X$ of points in $X$ that have a finite $\Gamma$ orbit under the action $T$ is essentially countable (in the sense that it is the union of a countable set and a null set). In particular, an ergodic probability-preserving action of a countable group admits a countable generator.
\end{cor}
\begin{proof}
The set $X_0 \subset X$ of points in $X$ that have a finite orbit is countable union $X_0 = \bigcup_{F \Subset \Gamma} X_F$, where the union is over finite subsets $F$ of the group $\Gamma$ and 
\[
X_F = \left\{ x \in X~:~ \{T_\gamma(x)\}_{\gamma \in \Gamma} = \{T_\gamma(x)\}_{\gamma \in F} \right\}.
\]
If $X_0$ is not essentially countable, then there exists a finite set $F \Subset \Gamma$ so that $X_F$ is not essentially countable. 
Thus, there exists a $\Gamma$-invariant $\mathcal{B}$-measurable set  $X' \subseteq X_F$ such that $\mu\mid_{X'}$ is a continuous measure. Suppose by contradiction that $\mathcal{P}=\{P_1,\ldots,P_n,\ldots\}$
is a countable generator for the action. Then the restriction of the $\sigma$-algebra $\mathcal{B}$ to the set $X'$ would coincide (up to completion by $\mu$) with restriction to $X'$ of the completion of the countable $\sigma$ algebra generated by the iterates of $\mathcal{P}$ by the finite set $F$.
This would imply that the restriction of $\mathcal{B}$ to $X'$ is the $\mu$ completion of a countable $\sigma$-algebra, in contradiction with the assumption that $\mu$ is continuous on $X'$.


Conversely, if $X_0$ is essentially countable, then $(X \setminus X_0)$ is a $T$-invariant set, and
\[(\tilde X,\tilde{\mathcal{B}}, \tilde \mu, \tilde T) =  (X\setminus X_0, \mathcal{B}_{X\setminus X_0}, \mu(\cdot\mid X \setminus X_0), T\mid_{X \setminus X_0})\] 
is a probability preserving transformation of a standard nonatomic probability space (because an atom with an infinite orbit would imply that the measure $\mu$ is infinite).

Thus $(\tilde X,\tilde{\mathcal{B}}, \tilde \mu, \tilde T)$
admits a countable generator $\tilde{\mathcal{P}}$ by \Cref{thm:countable_generator_group_action}. 
Write \[X_0 = \{ x_1,\ldots,x_n, \ldots\} \cup N,\] where $N$ is a null set.
Let $\mathcal{P}$ denote the countable partition of $X$ given by
\[ \mathcal{P} = \tilde{\mathcal{P}} \cup \{\{x_1\},\ldots,\{x_n\},\ldots\}\]
Then $\mathcal{P}$ is a countable generator for $(X,\mathcal{B},\mu,T)$.

If $(X,\mathcal{B},\mu,T)$ is ergodic, then either $\mu(X_0\setminus X)=0$ or $\mu(X_0)=0$. 
In the first case, $X$ is essentially a finite orbit, and so a countable generator clearly exists. In the second case, almost every orbit is infinite, so  by \Cref{thm:countable_generator_group_action}  $(X,
\mathcal{B},
\mu,T)$ admits a countable generator.
\end{proof}

We also remark that \Cref{thm:countable_generator_group_action} can be deduced from Seward's  work  \cite{MR3904452}, by repeated applications of  Seward's ``relative''  finite generator theorem.

\noindent\textbf{Acknowledgments: } We thank the referee  for pointing out some necessary corrections in the preliminary version and for other useful suggestions. This
research was partially supported by the Israel Science Foundation grant No. 985/23.

\bibliographystyle{amsplain}
\bibliography{lib}
\end{document}